\documentclass{amsart}
\usepackage{setspace}
\usepackage{a4}
\usepackage{amsthm}
\usepackage{latexsym}
\usepackage{amsfonts}
\usepackage{graphicx}
\usepackage{textcomp}
\usepackage{cite}
\usepackage{enumerate}
\usepackage{amssymb}
\usepackage{hyperref}
\usepackage{amsmath}
\usepackage{tikz}
\usepackage[mathscr]{euscript}
\usepackage{mathtools}
\newtheorem{theorem}{Theorem}[section]

\newtheorem{corollary}[theorem] {Corollary}
\newtheorem{definition}[theorem]{Definition}

\newtheorem{problem}[theorem]{Problem}

\newtheorem{question}[theorem]{Question}
\title{This is the title}
\usepackage{fancyhdr}

\begin{document}
\hrule\hrule\hrule\hrule\hrule
\vspace{0.3cm}	
\begin{center}
{\bf{FUNCTIONAL DONOHO-STARK-ELAD-BRUCKSTEIN-RICAUD-TORR\'{E}SANI UNCERTAINTY PRINCIPLE}}\\
\vspace{0.3cm}
\hrule\hrule\hrule\hrule\hrule
\vspace{0.3cm}
\textbf{K. MAHESH KRISHNA}\\
School of Mathematics and Natural Sciences\\
Chanakya University Global Campus\\
NH-648, Haraluru Village\\
Devanahalli Taluk, 	Bengaluru  North District\\
Karnataka State 562 110 India \\
Email: kmaheshak@gmail.com\\

Date: \today
\end{center}

\hrule\hrule
\vspace{0.5cm}
\textbf{Abstract}: Let $(\{f_j\}_{j=1}^n, \{\tau_j\}_{j=1}^n)$ and $(\{g_k\}_{k=1}^m, \{\omega_k\}_{k=1}^m)$ be  p-Schauder frames  for a finite dimensional Banach space $\mathcal{X}$. Then for every $x \in \mathcal{X}\setminus\{0\}$,  we show that 
\begin{align}\label{UE}
	\|\theta_f x\|_0^\frac{1}{p}\|\theta_g x\|_0^\frac{1}{q} \geq 	\frac{1}{\displaystyle\max_{1\leq j\leq n, 1\leq k\leq m}|f_j(\omega_k)|}\quad \text{and} 	\quad \|\theta_g x\|_0^\frac{1}{p}\|\theta_f x\|_0^\frac{1}{q}\geq \frac{1}{\displaystyle\max_{1\leq j\leq n, 1\leq k\leq m}|g_k(\tau_j)|},
\end{align}
where 
\begin{align*}
\theta_f: \mathcal{X} \ni x \mapsto (f_j(x) )_{j=1}^n \in \ell^p([n]); \quad \theta_g: \mathcal{X} \ni x \mapsto (g_k(x) )_{k=1}^m \in \ell^p([m])	
\end{align*}
and $q$ is the conjugate index of $p$. We call Inequality (\ref{UE}) as \textbf{Functional Donoho-Stark-Elad-Bruckstein-Ricaud-Torr\'{e}sani Uncertainty Principle}. Inequality (\ref{UE}) improves Ricaud-Torr\'{e}sani uncertainty principle \textit{[IEEE Trans. Inform. Theory, 2013]}. In particular, it improves  Elad-Bruckstein uncertainty principle \textit{[IEEE Trans. Inform. Theory, 2002]} and Donoho-Stark uncertainty principle \textit{[SIAM J. Appl. Math., 1989]}.



\hrule

\hrule
\section{Introduction}
While Heisenberg's uncertainty principle \cite{HEISENBERG} (English translation of original 1927 paper) is one of the greatest inequalities in the first half of the 20 century, Donoho-Stark uncertainty principle \cite{DONOHOSTARK} is one of the greatest inequalities in the second half of the 20 century. For $h \in \mathbb{C}^d$, let $\|h\|_0$ be the number of nonzero entries in $h$. Let $\hat{}: \mathbb{C}^d \to  \mathbb{C}^d$ be the Fourier transform.
\begin{theorem} (\textbf{Donoho-Stark Uncertainty Principle}) \cite{DONOHOSTARK} \label{DS}
	For every $d\in \mathbb{N}$, 
	\begin{align}\label{DSE}
  \left(\frac{\|h\|_0+\|\widehat{h}\|_0}{2}\right)^2\geq \|h\|_0\|\widehat{h}\|_0	\geq d, \quad \forall h \in \mathbb{C}^d\setminus \{0\}.
	\end{align}
\end{theorem}
In 2002, Elad and Bruckstein generalized Inequality (\ref{DSE})  to pairs of orthonormal bases \cite{ELADBRUCKSTEIN}. To state the result we need some notations. Given a collection $\{\tau_j\}_{j=1}^n$ in a finite dimensional Hilbert space $\mathcal{H}$ over $\mathbb{K}$ ($\mathbb{R}$ or $\mathbb{C}$), we define 
\begin{align*}
	\theta_\tau: \mathcal{H} \ni h \mapsto \theta_\tau h \coloneqq (\langle h, \tau_j\rangle)_{j=1}^n \in \mathbb{K} ^n.
\end{align*}
\begin{theorem} (\textbf{Elad-Bruckstein Uncertainty Principle}) \cite{ELADBRUCKSTEIN} \label{EB}
	Let $\{\tau_j\}_{j=1}^n$,  $\{\omega_j\}_{j=1}^n$ be two orthonormal bases for a  finite dimensional Hilbert space $\mathcal{H}$. Then 
	\begin{align*}
	\left(\frac{\|\theta_\tau h\|_0+\|\theta_\omega h\|_0}{2}\right)^2	\geq \|\theta_\tau h\|_0\|\theta_\omega h\|_0\geq \frac{1}{\displaystyle\max_{1\leq j, k \leq n}|\langle\tau_j, \omega_k \rangle|^2}, \quad \forall h \in \mathcal{H}\setminus \{0\}.
	\end{align*}
\end{theorem}
Note that Theorem \ref{DS} follows from Theorem \ref{EB}. In 2013, Ricaud and Torr\'{e}sani  showed that orthonormal bases in Theorem \ref{EB} can be improved to Parseval frames \cite{RICAUDTORRESANI}.
\begin{theorem} (\textbf{Ricaud-Torr\'{e}sani Uncertainty Principle}) \cite{RICAUDTORRESANI} \label{RT}
	Let $\{\tau_j\}_{j=1}^n$,  $\{\omega_j\}_{j=1}^n$ be two Parseval frames   for a  finite dimensional Hilbert space $\mathcal{H}$. Then 
\begin{align*}
	\left(\frac{\|\theta_\tau h\|_0+\|\theta_\omega h\|_0}{2}\right)^2	\geq \|\theta_\tau h\|_0\|\theta_\omega h\|_0\geq \frac{1}{\displaystyle\max_{1\leq j, k \leq n}|\langle\tau_j, \omega_k \rangle|^2}, \quad \forall h \in \mathcal{H}\setminus \{0\}.
		\end{align*}	
\end{theorem}
Following is famous  well known  application of Theorem \ref{EB}.
\begin{theorem} \cite{ELADBOOK, ELADBRUCKSTEIN} \label{APP}
	Let $\{\tau_j\}_{j=1}^n$,  $\{\omega_j\}_{j=1}^n$ be two orthonormal bases for a  finite dimensional Hilbert space $\mathcal{H}$. If $h \in \mathcal{H}$ can be written as 
	\begin{align*}
		h=\theta_\tau^*u=\theta_\tau^*v \quad \text{ for some } u, v \in \mathbb{K}^n, 
	\end{align*}
then 
\begin{align*}
	\left(\frac{\|u\|_0+\|v\|_0}{2}\right)^2	\geq	\|u\|_0\|v\|_0\geq \frac{1}{\displaystyle\max_{1\leq j, k \leq n}|\langle\tau_j, \omega_k \rangle|^2}.
\end{align*}
\end{theorem}

In this paper, we derive uncertainty principle for finite dimensional Banach spaces which contains Theorem \ref{RT} as a particular case. As  an, application we also derive Banach space version of Theorem \ref{APP}.

\textbf{Motivation}: Eventhough uncertainty principles for Hilbert spaces are in study from 100 years, uncertainty principles for abstract Banach spaces started in the beginning of 21 century only. Following are known uncertainty principles for Banach spaces. 
\begin{enumerate}
	\item  In 2006, Goh and Goodman extended Smith uncertainty principle from locally compact Abelian groups to Banach space \cite{GOHGOODMAN, SMITH}.
	\item In 2024, Krishna  extended Deutsch uncertainty principle from Hilbert spaces to Banach spaces \cite{KRISHNA3, DEUTSCH}.
	\item In 2024, Krishna extended Ghobber-Jaming uncertainty principle from Hilbert spaces to Banach spaces \cite{KRISHNA3, GHOBBERJAMING}.
\end{enumerate}

\section{Functional Donoho-Stark-Elad-Bruckstein-Ricaud-Torr\'{e}sani  Uncertainty Principle}
In the paper,   $\mathbb{K}$ denotes $\mathbb{C}$ or $\mathbb{R}$ and $\mathcal{X}$ denotes a  finite dimensional Banach space over $\mathbb{K}$. Identity operator on $\mathcal{X}$ is denoted by $I_\mathcal{X}$. Dual of $\mathcal{X}$ is denoted by $\mathcal{X}^*$. Whenever $1<p<\infty$, $q$ denotes conjugate index of $p$. For $d \in \mathbb{N}$, standard finite dimensional Banach space $\mathbb{K}^d$ over $\mathbb{K}$ equipped with standard $\|\cdot\|_p$ norm is denoted by $\ell^p([d])$. Canonical basis for $\mathbb{K}^d$ is denoted by $\{\delta_j\}_{j=1}^d$ and $\{\zeta_j\}_{j=1}^d$ be the coordinate functionals associated with $\{\delta_j\}_{j=1}^d$. We need the following variant of p-approximate Schauder frames defined by Krishna and Johnson in \cite{KRISHNAJOHNSON}. These are sub-classes of Schauder frames which have various applications, see \cite{CASAZZADILWORTH, DAUBECHIESDEVORE}.
\begin{definition}\label{PSF}
	Let $\mathcal{X}$  be a  finite dimensional Banach space over $\mathbb{K}$.   Let $\{\tau_j\}_{j=1}^n$ be a collection in  $\mathcal{X}$ and 	$\{f_j\}_{j=1}^n$ be a sequence in  $\mathcal{X}^*.$ The pair $(\{f_j\}_{j=1}^n, \{\tau_j\}_{j=1}^n)$ is said to be a \textbf{p-Schauder frame} ($1<p <\infty$) for $\mathcal{X}$ if  the following conditions hold.
	\begin{enumerate}[\upshape(i)]
		\item For every $x \in \mathcal{X}$, 
			\begin{align*}
			\|x\|^p=\sum_{j=1}^n|f_j(x)|^p.
		\end{align*}
		\item For every $x \in \mathcal{X}$, 
		\begin{align*}
		x=\sum_{j=1}^nf_j(x)\tau_j.
		\end{align*}
		\end{enumerate}
\end{definition}
We easily see that condition (i) in Definition \ref{PSF} says that the map 
\begin{align*}
\theta_f: \mathcal{X} \ni x \mapsto (f_j(x) )_{j=1}^n \in \ell^p([n])	
\end{align*}
is a linear isometry. Like Holub's characterization of  frames for Hilbert spaces \cite{HOLUB}, following theorem characterizes p-Schauder frames.
\begin{theorem}\label{THAFSCHAR}
	A pair  $(\{f_j\}_{j=1}^n, \{\tau_j\}_{j=1}^n)$ is a p-Schauder frame   for 	$\mathcal{X}$ if and only if 
	\begin{align*}
		f_j=\zeta_j U, \quad \tau_j=V\delta_j, \quad \forall j=1, \dots, n,
	\end{align*}  
	where $U:\mathcal{X} \rightarrow\ell^p([n])$, $ V: \ell^p([n])\to \mathcal{X}$ are linear operators such that $VU=I_\mathcal{X}$ and $U$ is an isometry.
\end{theorem}
Following is the crucial result of this paper. 
\begin{theorem} \label{FDS}(\textbf{Functional Donoho-Stark-Elad-Bruckstein-Ricaud-Torr\'{e}sani Uncertainty Principle}) 
Let $(\{f_j\}_{j=1}^n, \{\tau_j\}_{j=1}^n)$ and $(\{g_k\}_{k=1}^m, \{\omega_k\}_{k=1}^m)$ be  p-Schauder frames  for a Banach space $\mathcal{X}$. Then for every $x \in \mathcal{X}\setminus\{0\}$,  we have 
	\begin{align}\label{FDSUNCE}
		\|\theta_f x\|_0^\frac{1}{p}\|\theta_g x\|_0^\frac{1}{q} \geq 	\frac{1}{\displaystyle\max_{1\leq j\leq n, 1\leq k\leq m}|f_j(\omega_k)|}\quad \text{and} \quad \|\theta_g x\|_0^\frac{1}{p}\|\theta_f x\|_0^\frac{1}{q}\geq \frac{1}{\displaystyle\max_{1\leq j\leq n, 1\leq k\leq m}|g_k(\tau_j)|}.
	\end{align}
\end{theorem}
\begin{proof}
	Let $x \in \mathcal{X}\setminus\{0\}$ and $q$ be the conjugate index of $p$. First using $\theta_f$ is an isometry and later using $\theta_g$ is an isometry, we get 
	\begin{align*}
\|x\|^p&=\|\theta_fx\|^p=\sum_{j=1}^n|f_j(x)|^p=\sum_{j \in \operatorname{supp}(\theta_fx)}|f_j(x)|^p\\
&=\sum_{j\in \operatorname{supp}(\theta_fx)}\left|f_j\left(\sum_{k=1}^mg_k(x)\omega_k\right)\right|^p
=\sum_{j \in \operatorname{supp}(\theta_fx)}\left|\sum_{k=1}^mg_k(x)f_j(\omega_k)\right|^p\\
&=\sum_{j \in \operatorname{supp}(\theta_fx)}\left|\sum_{k \in  \operatorname{supp}(\theta_gx)}g_k(x)f_j(\omega_k)\right|^p\leq \sum_{j \in \operatorname{supp}(\theta_fx)}\left(\sum_{k \in  \operatorname{supp}(\theta_gx)}|g_k(x)f_j(\omega_k)|\right)^p\\
&\leq \left(\displaystyle\max_{1\leq j\leq n, 1\leq k\leq m}|f_j(\omega_k)|\right)^p\sum_{j \in \operatorname{supp}(\theta_fx)}\left(\sum_{k \in  \operatorname{supp}(\theta_gx)}|g_k(x)|\right)^p\\
&=\left(\displaystyle\max_{1\leq j\leq n, 1\leq k\leq m}|f_j(\omega_k)|\right)^p \|\theta_f x\|_0\left(\sum_{k \in  \operatorname{supp}(\theta_gx)}|g_k(x)|\right)^p\\
&\leq \left(\displaystyle\max_{1\leq j\leq n, 1\leq k\leq m}|f_j(\omega_k)|\right)^p \|\theta_f x\|_0\left(\sum_{k \in  \operatorname{supp}(\theta_gx)}|g_k(x)|^p\right)^\frac{p}{p}\left(\sum_{k \in  \operatorname{supp}(\theta_gx)}1^q\right)^\frac{p}{q}\\
&=\left(\displaystyle\max_{1\leq j\leq n, 1\leq k\leq m}|f_j(\omega_k)|\right)^p \|\theta_f x\|_0\|\theta_g x\|^p\|\theta_g x\|_0^\frac{p}{q}\\
&=\left(\displaystyle\max_{1\leq j\leq n, 1\leq k\leq m}|f_j(\omega_k)|\right)^p \|\theta_f x\|_0\|x\|^p\|\theta_g x\|_0^\frac{p}{q}.
	\end{align*}
Therefore 
\begin{align*}
	\frac{1}{\displaystyle\max_{1\leq j\leq n, 1\leq k\leq m}|f_j(\omega_k)|}\leq \|\theta_f x\|_0^\frac{1}{p}\|\theta_g x\|_0^\frac{1}{q}.
\end{align*}
On the other way, first using $\theta_g$ is an isometry and $\theta_f$ is an isometry, we get

\begin{align*}
\|x\|^p&=\|\theta_gx\|^p=\sum_{k=1}^m|g_k(x)|^p=\sum_{k \in \operatorname{supp}(\theta_gx)}|g_k(x)|^p\\
&=\sum_{k\in \operatorname{supp}(\theta_gx)}\left|g_k\left(\sum_{j=1}^nf_j(x)\tau_j\right)\right|^p
=\sum_{k \in \operatorname{supp}(\theta_gx)}\left|\sum_{j=1}^nf_j(x)g_k(\tau_j)\right|^p\\
&=\sum_{k \in \operatorname{supp}(\theta_gx)}\left|\sum_{j\in \operatorname{supp}(\theta_fx)}f_j(x)g_k(\tau_j)\right|^p\leq \sum_{k \in \operatorname{supp}(\theta_gx)}\left(\sum_{j\in \operatorname{supp}(\theta_fx)}|f_j(x)g_k(\tau_j)|\right)^p\\
&\leq \left(\displaystyle\max_{1\leq j\leq n, 1\leq k\leq m}|g_k(\tau_j)|\right)^p\sum_{k \in \operatorname{supp}(\theta_gx)}\left(\sum_{j\in \operatorname{supp}(\theta_fx)}|f_j(x)|\right)^p\\
&=\left(\displaystyle\max_{1\leq j\leq n, 1\leq k\leq m}|g_k(\tau_j)|\right)^p\|\theta_g x\|_0\left(\sum_{j\in \operatorname{supp}(\theta_fx)}|f_j(x)|\right)^p\\
&\leq \left(\displaystyle\max_{1\leq j\leq n, 1\leq k\leq m}|g_k(\tau_j)|\right)^p\|\theta_g x\|_0\left(\sum_{j\in \operatorname{supp}(\theta_fx)}|f_j(x)|^p\right)^\frac{p}{p}\left(\sum_{j\in \operatorname{supp}(\theta_fx)}1^q\right)^\frac{p}{q}\\
&=\left(\displaystyle\max_{1\leq j\leq n, 1\leq k\leq m}|g_k(\tau_j)|\right)^p\|\theta_g x\|_0\|\theta_f x\|^p\|\theta_f x\|_0^\frac{p}{q}\\
&=\left(\displaystyle\max_{1\leq j\leq n, 1\leq k\leq m}|g_k(\tau_j)|\right)^p\|\theta_g x\|_0\|x\|^p\|\theta_f x\|_0^\frac{p}{q}.
\end{align*}
Therefore 
\begin{align*}
\frac{1}{\displaystyle\max_{1\leq j\leq n, 1\leq k\leq m}|g_k(\tau_j)|}\leq \|\theta_g x\|_0^\frac{1}{p}\|\theta_f x\|_0^\frac{1}{q}.
\end{align*}
\end{proof}
\begin{corollary}
	Theorem \ref{RT} follows from Theorem \ref{FDS}.
\end{corollary}
\begin{proof}
	Let $\{\tau_j\}_{j=1}^n$,  $\{\omega_j\}_{j=1}^n$ be two Parseval frames   for a  finite dimensional Hilbert space $\mathcal{H}$.	Define 
	\begin{align*}
		f_j:\mathcal{H} \ni h \mapsto \langle h, \tau_j \rangle \in \mathbb{K}; \quad g_j:\mathcal{H} \ni h \mapsto \langle h, \omega_j \rangle \in \mathbb{K}, \quad \forall 1\leq j\leq n.
	\end{align*}
Then $p=q=2$, $\theta_\tau=\theta_f$, $\theta_\omega=\theta_g$ and 
\begin{align*}
|f_j(\omega_k)|=|\langle \omega_k, \tau_j \rangle |, \quad \forall 1\leq j, k \leq n.
\end{align*}
\end{proof}
Theorem  \ref{FDS}  brings the following question.
\begin{question}
	Given $p$, $m$, $n$ and a Banach space $\mathcal{X}$, for which pairs of p-Schauder frames $(\{f_j\}_{j=1}^n, \{\tau_j\}_{j=1}^n)$ and $(\{g_k\}_{k=1}^m, \{\omega_k\}_{k=1}^m)$, we have equality in Inequality (\ref{FDSUNCE})?
\end{question}
We now give Banach space version of Theorem \ref{APP}. For this, we need the notion of p-orthonormal bases for Banach spaces. 
\begin{definition} \cite{KRISHNA} \label{PONB}
	Let $\mathcal{X}$  be a  finite dimensional Banach space over $\mathbb{K}$.   Let $\{\tau_j\}_{j=1}^n$ be a basis for   $\mathcal{X}$ and 	let $\{f_j\}_{j=1}^n$ be the coordinate functionals associated with $\{\tau_j\}_{j=1}^n$. The pair $(\{f_j\}_{j=1}^n, \{\tau_j\}_{j=1}^n)$ is said to be a \textbf{p-orthonormal basis} ($1<p <\infty$) for $\mathcal{X}$ if  the following conditions hold.
	\begin{enumerate}[\upshape(i)]
		\item $\|f_j\|=\|\tau_j\|=1$ for all $1\leq j\leq n$.
		\item For every $(a_j)_{j=1}^n \in \mathbb{K}^n$, 
		\begin{align*}
			\left\|\sum_{j=1}^na_j\tau_j \right\|=\left(\sum_{j=1}^n|a_j|^p\right)^\frac{1}{p}.
		\end{align*}
	\end{enumerate}
\end{definition}
\begin{theorem}
	Let $(\{f_j\}_{j=1}^n, \{\tau_j\}_{j=1}^n)$,  $(\{g_k\}_{k=1}^n, \{\omega_k\}_{k=1}^n)$ be two p-orthonormal  bases for a  finite dimensional Banach space $\mathcal{X}$. If $x \in \mathcal{X}\setminus\{0\}$ can be written as 
\begin{align*}
	x=\theta_\tau u=\theta_\omega v \quad \text{ for some } u, v \in \mathbb{K}^n, 
\end{align*}
then 
\begin{align*}
		\|u\|_0^\frac{1}{p}\|v\|_0^\frac{1}{q} \geq 	\frac{1}{\displaystyle\max_{1\leq j\leq n, 1\leq k\leq m}|f_j(\omega_k)|}\quad \text{and} \quad \|v\|_0^\frac{1}{p}\|u\|_0^\frac{1}{q}\geq \frac{1}{\displaystyle\max_{1\leq j\leq n, 1\leq k\leq m}|g_k(\tau_j)|}.
\end{align*}
\end{theorem}
\begin{proof}
	Let $x \in \mathcal{X}\setminus\{0\}$ has representation 
	\begin{align*}
		x=\theta_\tau u=\theta_\omega v \quad \text{ for some } u, v \in \mathbb{K}^n.
	\end{align*}
	Since $(\{f_j\}_{j=1}^n, \{\tau_j\}_{j=1}^n)$ and   $(\{g_k\}_{k=1}^n, \{\omega_k\}_{k=1}^n)$ are p-orthonormal bases, we have 
	\begin{align*}
		\theta_f \theta_\tau =I_{\mathbb{K}^n}, \quad 	\theta_g \theta_\omega =I_{\mathbb{K}^n}.
	\end{align*}
 Using previous equations, we get 
 \begin{align*}
 	u=	\theta_f \theta_\tau u=\theta_f x, \quad v=	\theta_g \theta_\omega v=\theta_g x.
 \end{align*}
Therefore, by using Theorem \ref{FDS}
\begin{align*}
	&\|u\|_0^\frac{1}{p}\|v\|_0^\frac{1}{q} =	\|\theta_f x \|_0^\frac{1}{p}\|\theta_g x\|_0^\frac{1}{q}\geq 	\frac{1}{\displaystyle\max_{1\leq j\leq n, 1\leq k\leq m}|f_j(\omega_k)|},\\
	& \|v\|_0^\frac{1}{p}\|u\|_0^\frac{1}{q}=\|\theta_g x\|_0^\frac{1}{p}\|\theta_ f x\|_0^\frac{1}{q}\geq \frac{1}{\displaystyle\max_{1\leq j\leq n, 1\leq k\leq m}|g_k(\tau_j)|}.
\end{align*}
\end{proof}

\section{Open problem}
It is well-known  that Donoho-Stark uncertainty principle cannot be improved \cite{DONOHOSTARK}. However,  Donoho-Stark uncertainty principle gives (using AM-GM inequality) that 
\begin{align}\label{DSUPI}
	 \|h\|_0+\|\widehat{h}\|_0	\geq 2 \sqrt{d} , \quad \forall h \in \mathbb{C}^d\setminus \{0\}, \forall d \in \mathbb{N}.
\end{align}
In 2005, Tao improved Inequality (\ref{DSUPI}) in prime dimensions \cite{TAO}. 
\begin{theorem} \textbf{(Tao Uncertainty Principle)} \cite{TAO} 
Let $p \in \mathbb{N}$ be a prime number. Then 	
\begin{align*}
	 \|h\|_0+\|\widehat{h}\|_0	\geq p+1\geq 2 \sqrt{p}, \quad \forall h \in \mathbb{C}^p\setminus \{0\}.
\end{align*}
\end{theorem}
Motivated from Tao uncertainty principle, we ask the following problem.
\begin{problem}
	What is the version  Tao uncertainty principle for Banach spaces of prime dimension?
\end{problem}

 \bibliographystyle{plain}
 \bibliography{reference.bib}

\end{document}